\documentclass[12pt]{amsart}
\usepackage{amssymb,amsfonts,amsmath}
\usepackage{mathrsfs}
\usepackage{anysize}
\newtheorem{theorem}{Theorem}[section]
\newtheorem{lemma}[theorem]{Lemma}
\newtheorem{corollary}[theorem]{Corollary}

\newtheorem{example}[theorem]{Example}
\newcommand{\R}{{\mathbb{R}}}

\title[Piecewise contractions of the interval]{Asymptotically periodic \\ piecewise contractions of the interval}


\medskip
\begin{document}

\maketitle

\centerline{\scshape Arnaldo Nogueira\footnote{Partially supported by ANR Perturbations and BREUDS.}}
\smallskip
{\footnotesize
 \centerline{Aix-Marseille Universit\'e, Institut de Math\'ematiques de Luminy}
 \centerline {163, avenue de Luminy - Case 907, 13288 Marseille Cedex 9, France}
 \centerline{arnaldo.nogueira@univ-amu.fr}}

\medskip

\centerline{\scshape Benito Pires \footnote{Partially supported by FAPESP-BRAZIL 2013/12359-3.} and Rafael A. Rosales}
\smallskip

{\footnotesize
 \centerline{Departamento de Computa\c c\~ao e Matem\'atica, Faculdade de Filosofia, Ci\^encias e Letras}
 \centerline {Universidade de S\~ao Paulo, 14040-901, Ribeir\~ao Preto - SP, Brazil}
   \centerline{benito@usp.br, rrosales@usp.br} }

\bigskip

\marginsize{2.5cm}{2.5cm}{1cm}{2cm}

\begin{abstract} 
We consider the iterates of  a generic injective piecewise contraction of the interval
defined by a finite family of contractions.
 Let $\phi_i:[0,1]\to (0,1)$, $1\le i\le n$, be  $C^2$-diffeomorphisms with $\sup_{x\in (0,1)} \vert D\phi_i(x)\vert<1$ whose images $\phi_1([0,1]), \ldots , \phi_n([0,1])$ are pairwise disjoint. Let
 $0<x_1<\cdots<x_{n-1}<1$ and
let $I_1,\ldots , I_n$ be a partition of the interval $[0,1)$ into subintervals $I_i$ having interior
$(x_{i-1},x_i)$, where $x_0=0$ and $x_n=1$. Let $f_{x_1,\ldots,x_{n-1}}$ be the map given by $x\mapsto \phi_i(x)$ if $x\in I_i$, for $1\le i\le n$. Among other results we prove that for Lebesgue almost every  $(x_1,\ldots,x_{n-1})$, the piecewise contraction $f_{x_1,\ldots,x_{n-1}}$ is asymptotically periodic.
 \end{abstract}

\maketitle

\section{Introduction}

We say that $f:[0,1)\to [0,1)$ is a {\it piecewise contraction $($PC$)$ of $n$ intervals} if there exist $0<\kappa<1$ and a partition of $[0,1)$ into $n$ intervals $I_1,\ldots, I_n$ such that  $f\vert_{I_i}$ is $\kappa$-Lipschitz for every $1\le i\le n$. 

Much attention has been devoted to injective piecewise contractions of the interval because they appear as Poincar\'e maps. For instance, Poincar\'e maps induced by some Cherry flows on transverse intervals are topologically conjugate to injective piecewise contractions (see \cite{G1}). Injective PCs of the interval also arise as Poincar\'e maps of strange billiards governing switched server systems (see \cite{BB,CSR,MS}), and in the study of a certain class of outer billiards (see \cite{IJ}). 

In \cite{NP}, Nogueira and Pires proved that every injective PC $f$ of $n$ intervals has at most $n$ periodic orbits. Here
we are concerned with the long-term behavior of the iterates of $f$. For this purpose we recall two notions of periodicity.
A finite set $\gamma\subseteq [0,1)$ is a {\it periodic orbit} of $f$ if there exist $p\in [0,1)$ and an integer $k\ge 1$ such that $f^{k}(p)=p$ and $\gamma=\{p,f(p),\ldots, f^{k-1}(p)\}$. We say that $f$ is {\it asymptotically periodic} if there exist an integer $r\ge 1$ and periodic orbits $\gamma_1,\ldots, \gamma_r$ of $f$ such that
 $\omega(x)\in \{\gamma_1,\ldots,\gamma_r\}$ for every $x\in [0,1)$, where $\omega(x)=\bigcap_{m\ge 0}\overline{\bigcup_{k\ge m} f^k(x)}$ is the $\omega$-limit set of $x$. A weaker notion of periodicity is the following. Let $\varphi:[0,1)\to \{1,\ldots,n\}$ be the piecewise constant function defined by $\varphi(x)=i$ if $x\in I_i$. The {\it itinerary} of the point $x\in [0,1)$ is the sequence of digits $d_0,d_1,d_2,\ldots$ defined by $d_k=\varphi\left(f^k(x)\right)$. 
We say that {\it the itineraries of $f$ are eventually periodic} if the sequence $d_0,d_1,d_2,\ldots$ is eventually periodic for every $x\in [0,1)$. 
 
Our main result  asserts that generically injective PCs of $n$ intervals are asymptotically periodic and have at least one and at most $n$ periodic orbits, all of them attracting and stable. 
The existence of PCs  without periodic orbits  shows that not all piecewise contractions are asymptotically periodic.

In what follows we partition the set of  injective PCs of $n$ intervals into subsets $\mathscr{C}$, where each $\mathscr{C}$ is  determined by a fixed system of $n$ contractive maps of the interval. 
Let $A_1,\ldots, A_n$ be a sequence of pairwise disjoint compact subintervals of $(0,1)$ and 
$\phi_i:[0,1]\to A_i$ be a $C^2$-diffeomorphism with $\sup_{0<x<1} \vert D\phi_i(x)\vert<1$ for every $1\le i\le n$. Let 
\begin{equation}\label{omega1}
\Omega=\left\{(x_1,\ldots,x_{n-1})\mid 0<x_1<\cdots <x_{n-1}<1\right\},
\end{equation}
$x_0=0$ and $x_n=1$. Let $(x_1,\ldots,x_{n-1})\in \Omega$ and let $I_1,\ldots, I_n$ be a partition of $[0,1)$ into subintervals
$I_i$ having interior $(x_{i-1},x_i)$. 
Let $f_{x_1,\ldots,x_{n-1}}:[0,1)\to [0,1)$ be the PC of $n$ intervals  defined by
\begin{equation}\label{defphi}
f_{x_1,\ldots,x_{n-1}}(x)=\phi_i(x)\quad \textrm{for every}\quad x\in I_i\quad \textrm{and}\quad 1\le i\le n,   
\end{equation}
so its discontinuities are
$x_1,\ldots,x_{n-1}$ and its continuity intervals are $I_1,\ldots, I_n$.
We denote by $\mathscr{C}$ the set of all maps $f_{x_1,\ldots,x_{n-1}}$. Notice that $\displaystyle 2^{n-1}$ maps are associated to each $(x_1,\ldots,x_{n-1})\in\Omega$.

A periodic orbit $\gamma$ of  $f_{x_1,\ldots,x_{n-1}}$ is said to be {\it stable} if $\gamma\subseteq [0,1)\setminus \{x_1,\ldots,x_{n-1}\}$.  

Now we are able to state our main result. 

\begin{theorem}\label{main1} 
For  almost every $(x_1,\ldots,x_{n-1}) \in \Omega$, every 
$f_{x_1,\ldots,x_{n-1}}$ has  $1\le r\le n$ stable periodic orbits $\gamma_1,\ldots,\gamma_r$  such that  $\omega(x)\in \{\gamma_1,\ldots,\gamma_r\}$ for every $x\in [0,1)$.
\end{theorem}

Theorem \ref{main1} is closely related to \mbox{\cite[Theorem 2.2]{JB}} of Br\'emont, however our approach uses Ergodic Theory and is  different from the one presented in \cite{JB}.  Moreover, 
our definition of periodic orbit is the standard one and
 in our terminology  \mbox{\cite[Theorem 2.2]{JB}} states that the  itineraries of almost every injective PC  are  eventually periodic.
 There are piecewise contractions that, despite having only eventually periodic itineraries, fail to be
asymptotically periodic (see Example \ref{example}).
Therefore Theorem \ref{main1} is stronger than \mbox{\cite[Theorem 2.2]{JB}} and does not depend on the type of partition considered.

We call attention to two articles which are related to our work. In \cite{GT}, Gambaudo and Tresser listed all possible itineraries generated by a piecewise contraction $f:X_1\cup X_2\to X_1\cup X_2$, where
  $X_i$ is a complete metric space and $f\vert_{X_i}$ is a contraction for  $i=1,2$. Bruin and Deane \cite{BD} considered piecewise-linear contractions of $\R^2$ and proved that a large class of these maps are asymptotically periodic.

 This article is organised in the following way. For the sake of clarity, the proof of Theorem \ref{main1} is split up into the next two sections. Section 2 includes an alternative proof of \cite[Theorem 2.2]{JB} whose contents will be essential to Section 3, where we prove our main result. In Section 4 we provide an example of a PC of two intervals without periodic orbit whose itineraries are eventually periodic.

Throughout this article all metrical statements concern the Lebesgue measure.

\section{Eventually periodic itineraries}

The first part of the proof of our main theorem leads to the following:

\begin{theorem}\label{main2}  
For almost every $(x_1,\ldots,x_{n-1}) \in \Omega$,
the  itineraries of  every $f_{x_1,\ldots,x_{n-1}}$ are eventually periodic.
\end{theorem} 

If the partition of the interval $[0,1)$ equals   $I_1=[x_0,x_1), \ldots,$ $I_n=[x_{n-1},x_n)$, then
Theorem \ref{main2}  and \cite[Theorem 2.2]{JB} are equivalent.
The approach followed in \cite{JB}  uses a series of distortion lemmas to estimate the  images of the continuity intervals $I_i$ by $f\in\mathscr{C}$. 
Here,
instead of looking at the  images of the partition intervals, we look at the  preimages of the discontinuities of $f$.  
This approach was motivated by the analysis of the switched server system considered in \cite{CSR}. It is convenient because it allows to define an expanding piecewise smooth Markov map $g:[0,1]\to [0,1]$ which is a left inverse of every $f\in\mathscr{C}$.
Hence, the study of the map $g$ provides some insight into the behaviour of a typical member of $\mathscr{C}$. By the ergodic properties of $g$ given by the  Folklore Theorem,  its forward orbit $x,g(x),g^2(x),\ldots$  is dense in $[0,1]$ for almost every $x$. In terms of $\mathscr{C}$, this means that for almost every $(x_1,\ldots,x_{n-1})$, the backward orbit of $f=f_{x_1,\ldots,x_{n-1}}$ of every discontinuity $x_i$ is the dense forward $g$-orbit of $x_i$ until it hits a {\it gap}, namely a maximal subinterval in $[0,1)\setminus f\left([0,1)\right)$. Hence $f$ admits an invariant quasi-partition which ensures that the
itineraries of $f$ are periodic.

Let  $f=f_{x_1,\ldots,x_{n-1}} \in \mathscr{C}$.  A collection $\mathscr{P}=\{J_{\ell}\}_{\ell=1}^m$ of pairwise disjoint open subintervals of $(0,1)$ is an {\it invariant quasi-partition under} $f$
if the following is satisfied:
\begin{itemize}
\item [(P1)] The set $H=(0,1)\setminus \cup_{\ell=1}^m J_\ell$ is finite and contains $\{x_1,\ldots,x_{n-1}\}$;
\item [(P2)] There exists a map $\tau:\{1,\ldots,m\}\to \{1,\ldots,m\}$ such that $f\left(J_\ell \right)\subseteq J_{\tau(\ell)}$ for every $1\le \ell\le m$.
\end{itemize}

\begin{lemma}\label{ipartition} 
If $f$ admits an invariant quasi-partition, then its  itineraries  are eventually periodic.
\end{lemma}
\begin{proof} 
Let $\mathscr{P}=\{J_{\ell}\}_{\ell=1}^m$ be an invariant quasi-partition under $f$.
Let $1 \le \ell_0 \le m$ and $\{\ell_k\}_{k=0}^\infty$ be the sequence defined recursively by  
$\ell_{k+1}=\tau (\ell_k)$ for every  $k\ge 0$, where $\tau$ is given by $(P2)$. 
It is elementary that  the sequence $\{\ell_k\}_{k=0}^\infty$ is eventually periodic.
By $(P1)$, there exists a unique map $\eta: \{1,\ldots ,m\} \to \{1,\ldots ,n\}$ satisfying $ J_{\ell} \subseteq I_{\eta(\ell)}$ for every $1 \le \ell \le m$,
thus the sequence $\{\eta(\ell_k)\}_{k=0}^\infty$ is eventually periodic. 
By definition, the itinerary of any $x \in  J_{\ell_0}$ is the sequence $\{\eta(\ell_k)\}_{k=0}^\infty$.

Now let $x\in H$ which is a finite set by $(P1)$. If $O_f(x)\subseteq H$, then, as $f$ is injective, the orbit of $x$ is  periodic, hence its itinerary is   periodic. Otherwise, there exist  $1\le \ell_0 \le m$ and $k\ge 1$ such that $f^k(x)\in J_{\ell_0}$. By the above, the itinerary of $f^k(x)$ is eventually periodic and so is that of $x$. This proves the lemma.
\end{proof}

Next we state a particular version of the Folklore Theorem, a  known result in one-dimensional dynamics. It
holds for expanding Markov maps with a finite or countable partition satisfying the bounded distortion property $\sup_i\sup_{x,y\in K_{i}}{\vert D^2 f(x)\vert}\big/{\vert Df(y)\vert^2}<\infty$. In the version below, this property follows automatically from the hypothesis (i)  because the partition is finite. A proof of the result can be found in \cite[pp. 305--310]{RA1}, see also \cite{RB}. Let $\overline{A}$ denote the topological closure of any $A\subseteq \R$.

\begin{theorem}[Folklore Theorem]\label{folklore} 
Let $K_1,\ldots,K_d$ be a partition of $[0,1]$ into $d$ intervals, $c>1$ and $g:[0,1]\to [0,1]$ be a map
such that for every $1\le i\le d$ the following holds
\begin{itemize}
\item [(i)] $g\vert_{K_i}$ extends to a $C^2$-diffeomorphism from $\overline{K_i}$ onto $[0,1]$;
\item [(ii)] $ \vert Dg(x)\vert\ge c$, for every $x$ in the interior of $K_i$.
\end{itemize}
Then $g$ has an invariant ergodic probability measure $\mu$ equivalent to the Lebesgue measure.
\end{theorem}

\begin{corollary}\label{bet} Let $g$ and  $\mu$ be as in Theorem \ref{folklore}, then for almost every $x\in [0,1]$ the orbit of $x$, $O_g(x)=\{x,g(x),g^2(x),\ldots\}$, is dense in $[0,1]$.
\end{corollary}
\begin{proof} 
Let  $1\leq \ell \leq m$ and $J_{\ell,m}=(\frac{\ell-1}{m},\frac{\ell}{m})$.
 As the measure $\mu$ is equivalent to the Lebesgue measure, $\mu(J_{\ell,m})>0$, thus,  
as $g$ is ergodic, $D_{\ell,m}=\cup_{k=0}^{\infty} g^{-k}(J_{\ell,m})=[0,1]$ almost surely. 
So $J=\cap_{m=1}^{\infty}\cap_{ \ell=1}^m  D_{\ell,m}=[0,1]$ almost surely and $\overline{O_g(x)}=[0,1]$ for every $x\in J$.
\end{proof} 

We call a map $g:[0,1]\to [0,1]$ satisfying the hypotheses of  Theorem \ref{folklore} an {\it expanding piecewise smooth map}.
The next result and Corollary \ref{bet}  are of paramount importance for proving Theorem \ref{main2}.

\begin{lemma}\label{existeg} 
There exists an expanding piecewise smooth map $g:[0,1]\to [0,1]$ such that
$g\left( f(x)\right)=x$ for every $x\in [0,1)$ and every $f\in\mathscr{C}$.
\end{lemma}
\begin{proof} 
The map
$\phi_i:[0,1]\to A_i$ is a contractive \mbox{$C^2$-diffeomorphism}, thus its inverse $\phi_i^{-1}:A_i\to [0,1]$ is a $C^2$-diffeomorphism satisfying $c_i=\inf_{x\in A_i} \vert D\phi_i^{-1}(x)\vert>1$. 
Moreover, the sets $A_1$, \ldots, $A_n$ are pairwise disjoint compact subintervals of $(0,1)$, thus $[0,1)\setminus \bigcup_{i=1}^n A_i$ is the union of the pairwise disjoint  intervals $B_1,\ldots,B_{n+1}$. 
Denote by 
$L_j$  the unique increasing affine map  from $\overline{B_j}$ onto $[0,1]$. Note that its slope is at least 
$\widetilde{c}=\min_j\vert  B_j \vert^{-1}$, where $\vert  B_j \vert$ is the length of $B_j$.
 By definition
$A_1,\ldots , A_n,B_1,\ldots ,B_{n+1} $ is a partition of $[0,1)$. 
Let $g:[0,1]\to [0,1]$ be the map defined by
\begin{equation}\label{defg}
g(x)=
\begin{cases}
\phi_i^{-1}(x) & \textrm{if}  \quad x\in A_i\quad \textrm{for some}\quad  1\le i\le n \\ 
L_j(x)  & \textrm{if} \quad  x\in B_j\quad \textrm{for some}\quad  1\le j\le n+1\end{cases}.
\end{equation}
 By construction, $g$ is an expanding piecewise smooth map which satisfies all the conditions of Theorem \ref{folklore} with $c=\min\, \{c_1,\ldots,c_{n},\widetilde{c}\}$. 
Let $f=f_{x_1,\ldots,x_{n-1}}\in \mathscr{C}$, then  
  $f(x)=\phi_i(x)\in A_i$,  for every $x\in I_i$,  thus $g\left(f(x)\right)=g\left(\phi_i(x)\right)=\phi_i^{-1}\left(\phi_i(x)\right)=x$.
\end{proof}

\begin{proof}[{\bf Proof of Theorem \ref{main2}}] 
Let $g:[0,1]\to [0,1]$ be the map defined in $(2.1)$, thus $g$ satisfies the hypotheses of Theorem \ref{folklore}.
Let $J=\{x\in [0,1)\mid \overline{O_g(x)}=[0,1]\}$, thus
by Corollary \ref{bet}, $J=[0,1]$  almost surely.
Let $\Omega$ be the set defined in $(\ref{omega1})$ and 
\begin{equation}\label{defU}
U=\Omega \cap J^{n-1},\,\, \textrm{thus $U=\Omega$  almost surely}.
\end{equation}
    
Let $(x_1,\ldots,x_{n-1})\in U$, hereafter we will prove that $f=f_{x_1,\ldots,x_{n-1}} \in\mathscr{C}$ admits an invariant quasi-partition.

We denote by $f^{-1}$ the inverse of the map $x\in [0,1) \mapsto f(x) \in f\left( [0,1)\right)$, thus its domain equals $f\left( [0,1)\right)$. If $f^{-1}$ can be iterated $k\geq 0$ times at $x$, 
we  denote by $f^{-k}(x)$ the $k$th iterate of $f^{-1}$ at $x$.
Let 
$\displaystyle G=[0,1)\setminus f\left( [0,1)\right),  
S=\bigcup_{k=0}^\infty f^k(G)$ and $ W=[0,1)\setminus S$.\\

\noindent Claim A.
The set $G$  has a nonempty interior and $G, f(G), f^2(G),\ldots$ are pairwise disjoint.

The image $f([0,1))=\cup_{i=1}^n f(I_i) \subseteq  \cup_{i=1}^n A_i$ and
$\cup_{i=1}^n A_i$ is a compact subset of the interval $(0,1)$, hence $ (0,1) \setminus \cup_{i=1}^n A_i$ is a nonempty open subset of $G$ which proves the first assertion.
Now note that $G\cap  f\left ([0,1)\right)=\emptyset$ and $f^k(G)\subseteq f\left ([0,1)\right)$, hence $f^k(G)\cap G=\emptyset$ for every integer $k\ge 1$. Therefore, as $f$ is injective,  the sets $G, f(G), f^2(G),\ldots$ are pairwise disjoint. 
This proves Claim A.\\

\noindent 
Claim B. $W\subseteq f\left( [0,1)\right)$ and $f^{-1}(W) \subseteq W$.

We have that $W=[0,1)\setminus \bigcup_{k=0}^\infty f^k(G)\subseteq [0,1)\setminus G=f\left([0,1)\right)$.
Moreover,   the fact that $f^{-1}(G)=\emptyset$ and $f$ is injective imply that $f^{-1}(S)=S$, so $f^{-1}(W)\subseteq W$.\\

\noindent 
Claim C. $f^{-1}=g\vert_{f\left( [0,1)\right)}$.

By Lemma \ref{existeg}, $g(f(x))=x=f^{-1}(f(x))$ for every $x\in [0,1)$.\\

\noindent Claim D. $\{x_1,\ldots,x_{n-1}\}\subseteq S$.

By contradiction, assume $x_i\in W$ for some $i$, thus
$O_{f^{-1}}(x_i)$ is well defined by Claim B. 
By Claim C, $O_{f^{-1}}(x_i)=O_g(x_i)$. As $x_i \in J$, $O_g(x_i)$ is dense in $[0,1)$ and so is $O_{f^{-1}}(x_i)$. As $G$ contains a nonempty open set, there exists $q_i\ge 0$ such that $f^{-q_i}(x_i)$ is in the interior of $G$ or equivalently, $x_i\in f^{q_i}(G)$. Thus $x_i\in S$, i.e. $x_i\not\in W$, which is a contradiction.\\

 \noindent Claim E. 
For every $1\le i \le n-1$, there exists a unique $q_i\ge 0$ such that $x_i\in f^{q_i}(G)$.

By Claim D, there exists $q_i\ge 0$ such that $x_i\in f^{q_i}(G)$ for every $1\le i \le n-1$.
By Claim A, $G, f(G), f^2(G),\ldots$ are pairwise disjoint, thus $q_i$ is unique.\\

\noindent Claim F.
Let $Q_i=\{ f^{-\ell}(x_i)\mid 0\le \ell \le q_i\} $, thus $f^{-1}(Q_i)\subseteq Q_i$ for every $1\le i\le n-1$.

     This claim follows from the fact that $f^{-q_i}(x_i)\in G$ and $f^{-1}(G)=\emptyset$.\\

\noindent Claim G. Let $ J_1,\ldots , J_m$ be the  connected components of $(0,1)\setminus \cup_{i=1}^{n-1} Q_i$, then $\mathscr{P}=\{ J_\ell\}_{\ell=1}^m$ is an invariant quasi-partition under $f$.

The collection $\mathscr{P}$ is a quasi-partition which fulfills condition $(P1)$. Let us show that it satisfies   condition $(P2)$.
Let  $1 \le \ell_0 \le m$, thus by  the definition of $ J_{\ell_0}$, $\{x_1,\ldots , x_{n-1} \} \cap J_{\ell_0}=\emptyset$, which implies that $f( J_{\ell_0})$ is an open interval. We claim that $f(J_{\ell_0})$ is contained in the open set $\bigcup_{\ell=1}^{m}J_{\ell}$. Suppose that this is false,
then $f(J_{\ell_0})\cap Q_i\neq\emptyset$ for some $1\le i\le n-1$. By the injectivity of $f$, $J_{\ell_0}\cap f^{-1}(Q_i)\neq\emptyset$.
By Claim F, $J_{\ell_0}\cap Q_i\neq\emptyset$, which contradicts the definition of $\mathscr{P}$. This proves Claim G. \\

The proof of Theorem \ref{main2} now follows from Claim G  together with Lemma \ref{ipartition}.
 \end{proof}

  \section{Asymptotically periodic orbits}
    
  We keep the notation used in Section $2$. 
  Let  $(x_1,\ldots,x_{n-1})\in U$, then
  the  itineraries  of  $f_{x_1,\ldots,x_{n-1}}$ are eventually periodic, by Theorem \ref{main2}. We will prove bellow that $f_{x_1,\ldots,x_{n-1}}$ is indeed asymptotically periodic for almost every $(x_1,\ldots,x_{n-1})\in U$.

By the definition, $0 < f(x) <1$, for every $x\in [0,1)$. 
  We say that $f_{x_1,\ldots,x_{n-1}} \in \mathscr{C}$ has a {\it g-connection} if there exist $1\le i\le n-1$, $0\le j\le n$ and $k\ge 1$ such that $g^k(x_i)=x_j$.

\begin{lemma}\label{N=0} 
For almost every $(x_1,\ldots,x_{n-1}) \in \Omega$, every
$f_{x_1,\ldots,x_{n-1}}$ has no $g$-connection.  
    \end{lemma}
    \begin{proof} 
The set $N_{i,j,k}=\left\{(x_1,\ldots,x_{n-1})\in \Omega\mid x_{j}=g^{k}(x_i)\right\}$ 
is the graph of a smooth function, for every $1\le i,j\le  n-1$ and $k\ge 1$, thus it is a null set.
Let $N_{i,0,k}=\left\{(x_1,\ldots,x_{n-1})\in\Omega\mid x_{0}=0=g^{k}(x_i)\right\}$ and 
 $N_{i,n,k}=\left\{(x_1,\ldots,x_{n-1})\in\Omega\mid x_{n}=1=g^{k}(x_i)\right\}$. 
Let us prove that both are null sets.
Let $K= \cup_{k=0}^{\infty} g^{-k}(\{0,1\})$. Since $g$ is a finite-to-one map, the set $K$ is countable. 
Let $\pi_i: \Omega\to [0,1)$ be the projection $(x_1,\ldots,x_{n-1}) \mapsto x_i$, therefore 
$\pi_i (N_{i,0,k}) \subseteq K$ and $\pi_i(N_{i,n,k}) \subseteq K$, hence $N_{i,0,k}$ and $N_{i,n,k}$ are null sets. We have proved that the set 
$\displaystyle
\label{defS} \bigcup_{i=1}^{n-1}\bigcup_{j=0}^{n}\bigcup_{k=1}^{\infty} N_{i,j,k}$ 
 is a null set. This proves the lemma.
\end{proof}

\begin{proof}[{\bf Proof of Theorem \ref{main1}}]
    
 Hereafter, assume that $f=f_{x_1,\ldots,x_{n-1}}$ has no $g$-connection.\\

\noindent Claim H. For every $x\in\partial G$, $g(x)\in \{x_0,x_1,\ldots,x_{n}\}$.

The boundary of $G$ equals $\cup_{i=1}^n \partial f(I_i) \cup \{0,1\}$. 
By the definition of $g$ in (2.1), $g(0)=0=x_0$ and $g(1)=1=x_n$. 
Let $a,b\in  \partial f(I_i)$ with $a<b$.
Without loss of generality, assume that $f\vert_{I_i}$ is increasing, then $a=\lim_{\epsilon\to 0+} f(x_{i-1}+\epsilon)$. Moreover, $a\in  \overline{f\left(I_i\right)}\subseteq A_i$. 
By (2.1), $g$  is continuous on $A_i$ and $g(f(x))=x$ for every $x\in I_i$, so
        $$
    g(a)=g\left(\lim_{\epsilon\to 0^+}  f(x_{i-1}+\epsilon)\right)=\lim_{\epsilon\to 0^+} g\left(f(x_{i-1}+\epsilon) \right)=\lim_{\epsilon\to 0+} (x_{i-1}+\epsilon)=x_{i-1}.
    $$
Analogously $g(b)=x_i$  which proves Claim H.\\

\noindent Claim I. $f^{-k}(x_i)=g^k(x_i)$ for every $1\le i\le n-1$ and $0\le k\le q_i$.

By Claim E, $f^{-k}(x_i)$ is well defined  for every $0\le k\le q_i$. By the definition of $f^{-k}$ and Claim C, 
 $f^{-k}(x_i)=g^{k}(x_i)$,  which proves Claim I.\\

\noindent 
Claim J. $f^{-q_i}(x_i)$ belongs to the interior of $G$ for every $1\le i\le n-1$.
     
 By the definition of $q_i$, $f^{-q_i}(x_i)\in G$. By contradiction assume $f^{-q_i}(x_i)\in \partial G$, then, by Claim H, $g\left( f^{-q_i}(x_i)\right)=x_j$, where $0\le j\le n$. On the other hand, by Claim I,
     $f^{-q_i}(x_i)=g^{q_i}(x_i)$. Therefore, $g^{q_i+1}(x_i)=x_j$ which contradicts the fact that  $f$ has no $g$-connection. This proves Claim J.\\



\noindent 
Claim K. Let $q= \max \,\{q_1, \ldots , q_{n-1}\}$ and $E$ be the interior of  $\bigcup_{k=0}^q f^k(G)$, then 
$$
\bigcup _{i=1}^{n-1} Q_i\subseteq E.
$$ 

Let $1\le i \le n-1$. If $q_i=0$, by Claim J, $x_i$ belongs to the interior of $G$ which is contained in the open set $E$. Now assume that  $q_i\ge 1$ and let $1\le k \le q_i$. By Claim I, $f^{-k}(x_i)=g^k(x_i)$. As by assumption $f$ has no $g$-connection, we have that 
\begin{equation}\label{xon}
f^{-k}(x_i)\notin \{x_1,\ldots,x_{n}\}\quad\textrm{for every}\quad 1\le k\le q_i.
\end{equation}
Let us prove by recurrence that $Q_i\subseteq E$. By Claim J, $f^{-q_i}(x_i)$ belongs to the interior of $G\subseteq E$. Now let us consider $k=q_i-1$. 
By (\ref{xon}), $f^{-q_i}(x_i)\not\in \{x_1,\ldots,x_{n-1}\}$, thus $f$ is locally continuous at $f^{-q_i}(x_i)$ and injective, which implies that 
the point $f^{-k}(x_i)=f(f^{-q_i}(x_i))$ belongs to the interior of $f(G)$, so belongs to $E$. Moreover, by (\ref{xon}), if $k\ge 1$ then
$f^{-k}(x_i)\notin \{x_1,\ldots,x_{n}\}$ and the reasoning can be applied once more.\\
     
\noindent 
Claim L. $f^p(E)\cap E=\emptyset$ for every integer $p>q$.
     
It follows  from Claim A and the definition of $E$ that, for $p>q$, 
$$
E \cap f^p(E) \subseteq
\bigcup_{k=0}^q f^k(G) \cap \bigcup_{m=p}^{p+q} f^m(G) = \emptyset.
$$

Now we will use the above claims to conclude the proof of the theorem.
      
 Let $\mathscr{P}=\{J_{\ell}\}_{\ell=1}^m$ be the invariant quasi-partition under $f$ given by Claim G. 
Let $1\le \ell_0 \le m$ and $x\in  J_{\ell_0}$. In the proof of Lemma \ref{ipartition}, it is showed that the itinerary of $x$ in  $\mathscr{P}$, $\{\ell_k\}_{k=0}^{\infty}$,   is eventualy periodic. Therefore there exist $s \ge 1$ and  $p>q$ such that 
  $\ell_{s}=\ell_{s+p}$, where $q$ is given by Claim K. As $\mathscr{P}$ is invariant under $f$,
  $f^{p}(J_{\ell_s})\subseteq J_{\ell_s}$. 
Hence, if $J_{\ell_s}=(a,b)$, there exist $a\le c<d\le b\le 1$ such that $f^{2p}\left((a,b)\right)= (c,d)$. We claim that  $c>a$.
Assume by contradiction that $c=a$. We have that $a\in\partial J_{\ell_s}\subseteq \{0,1\}\cup\bigcup_{i=1}^{n-1} Q_i$ and $a<1$.   According to Claim K, there exists $\epsilon > 0$ such that $(a,a+ \varepsilon) \subseteq E\cap J_{\ell_s} $, 
thus, as $f^{2p}\vert_{(a,b)}$ is an increasing contractive map,  $f^{2p}((a, a+\varepsilon))\subseteq (a, a+\varepsilon)$.
We conclude that 
$f^{2p}(E)\cap E\neq \emptyset$,
which contradicts Claim L. This proves the claim.
Analogously $d<b$. Thus $f^{2p}\left((a,b)\right)= (c,d)$, where $f^{2p}\vert_{(a,b)}$ is a continuous contraction and $a< c<d < b$, therefore $f^{2p}$ has a unique fixed point $y\in (a,b)$.
Notice that $\gamma=O_f(y)$ is a periodic orbit of $f$, moreover, it is clear that $\omega(x)=\gamma$ for every $x\in J_{\ell_0}$.

Now we consider the case in which $x\in [0,1) \setminus\cup_{\ell=1}^m J_{\ell}=\{0\} \cup \bigcup _{i=1}^{n-1} Q_i$.
By the proof of Lemma \ref{ipartition}, either  $O_f(x)$ is in the finite set $[0,1) \setminus\cup_{\ell=1}^m J_{\ell}$ 
or there exists  $k\geq 1$ such that $f^k(x)\in\cup_{\ell=1}^m J_{\ell}$.
The first case is impossible because $f$
 has no $g$-connection and $f(0)>0$. In the latter case, by the above, 
$\omega(x)$ is a periodic orbit. 
We have proved that there exist $r\ge 1$ stable periodic orbits $\gamma_1,\ldots, \gamma_r$ with
  $\omega(x) \in\{\gamma_1,\ldots,\gamma_r\}$, for every $x\in [0,1)$. By   \cite[Theorem 1.1]{NP}, we have that $r\le n$.  This concludes the proof of Theorem \ref{main1}.
 \end{proof}

\section{Final Remarks}

We present here an example to expose the difference between the two periodicity notions considered so far.

\begin{example}\label{example} Let $f_1:[0,1)\to [0,1)$ be the PC of two intervals defined by
$$
f_1(x)=
\begin{cases}
\phantom{-}\dfrac14+\dfrac{x}{2}\quad{\rm if}\quad x\in I_1=[0,1/2)\\[1em]
-\dfrac{1}{4}+\dfrac{x}{2}\quad{\rm if}\quad x\in I_2=[1/2,1)
\end{cases}.
$$
\end{example}
\noindent
Notice that $f_1(I_1) \cup f_1(I_2)\subseteq I_1$, thus $f_1$ admits only the itineraries 1111\ldots and 2111\ldots which are eventually periodic. On the other hand, for every $x\in [0,1)$, the sequence $f^2(x),f^3(x),\ldots$ is strictly increasing and converges to $\frac12$, thus $f_1$ has no periodic orbit and the one-point set $\{\frac12\}$ works as a global attractor of $f_1$. 
Therefore $f_1$ is not asymptotically periodic. 

Notice that redefining $f_1$ at $x_1=\frac12$ yields the asymptotically periodic map
$$
f_2(x)=
\begin{cases}
\phantom{-}\dfrac14+\dfrac{x}{2}\quad\textrm{if}\quad x\in [0,1/2]\\[1em]
-\dfrac{1}{4}+\dfrac{x}{2}\quad \textrm{if}\quad x\in (1/2,1)
\end{cases}.
$$
The map $f_2$ is not stable in the $C^0$-uniform topology: the fixed-point $x_1=\frac12$ is easily destroyed by a perturbation of $f_2$. Precisely, for every $0<\epsilon<\frac14$, the $2$-interval $PC$  $[0,1)\to [0,1)$ defined by $x \mapsto f_2(x)+\epsilon$ has no fixed-point (i.e. no 1-periodic orbit). 

We observe that a piecewise contraction $f$ given by Theorem \ref{main1} is asymptotically periodic in a stable way: if $\gamma$ is a $k$-periodic orbit of $f$, then any PC  $C^0$-close to $f$ has a $k$-periodic orbit close to $\gamma$.

\medskip

\noindent
{\it Acknowledgements} The first-named author would like to thank the 
Departamento de Computa\c c\~ao e Matem\'atica, 
 USP, Ribeir\~ao Preto, Brazil and 
De Giorgi Center, Italy, for hospitality while this work was done.

\end{document}